\documentclass{amsart}
\usepackage{amssymb}
\usepackage{amscd}

\newtheorem{thm}{Theorem}[section]

\newtheorem{lemma}[thm]{Lemma}

\newtheorem{conjecture}{Conjecture}
\theoremstyle{definition}

\theoremstyle{remark}

\newtheorem*{acknowledgements}{Acknowledgements}
\numberwithin{equation}{section}
\newcommand{\Z}{\mathbb{Z}}

\newcommand{\M}{\widetilde{M}}

\newcommand{\del}{\partial}

\begin{document}

\title [topology of manifolds with positive isotropic curvature] {On the topology of manifolds with\\ positive isotropic curvature}

\author{Siddartha Gadgil}
\address{department of mathematics,
Indian Institute of Science, Bangalore 560012, India}
\email{gadgil@math.iisc.ernet.in}

\author{harish seshadri}
\address{department of mathematics,
Indian Institute of Science, Bangalore 560012, India}
\email{harish@math.iisc.ernet.in}

\begin{abstract}
We show that a closed orientable Riemannian $n$-manifold, $n \ge
5$, with positive isotropic curvature and free fundamental group
is homeomorphic to the connected sum of copies of $S^{n-1}\times
S^1$.
\end{abstract}


\maketitle
% ----------------------------------------------------------------
\section{Introduction}
 Let $(M,g)$ be a closed, orientable,
Riemannian manifold with positive isotropic curvature. By
~\cite{mm}, if $M$ is simply-connected then $M$ is homeomorphic
to a sphere of the same dimension. We shall generalise this to the
case when the fundamental group of $M$ is a free group.

\begin{thm}\label{posisot}
Let $M$ be a closed, orientable Riemannian $n$-manifold with
positive isotropic curvature. Suppose that $\pi_1(M)$ is a free
group on $k$ generators. Then, if $n\neq 4$ or $k=1$ (i.e.
$\pi_1(M)=\Z$), $M$ is homeomorphic to the connected sum of $k$
copies of $S^{n-1}\times S^1$.
\end{thm}

We note that a conjecture of M. Gromov(~\cite{gr} Section 3
(b)) and A. Fraser~\cite{fr}, based on the work of Micallef-Wang  ~\cite{mw},  states  that any compact manifold with
positive isotropic curvature has a finite cover satisfying our
hypothesis.

\begin{conjecture}[M. Gromov-A. Fraser]
$\pi_1(M)$ is virtually free, i.e., it is a finite extension of a free
group.
\end{conjecture}

It is known by the work of A. Fraser ~\cite{fr} and A. Fraser and
J. Wolfson ~\cite{fr2} that $\pi_1(M)$ does not contain any
subgroup isomorphic to the fundamental group of a closed surface
of genus at least one.

Our starting point is the following fundamental result of M.
Micallef and J. Moore~\cite{mm}.

\begin{thm}[M. Micallef-J. Moore]
Suppose $M$ is a closed manifold with positive isotropic
curvature. Then $\pi_i(M)=0$ for $2 \le \ i \ \le \frac {n}{2}$.
\end{thm}

It is clear that the following purely topological result, together with the Micallef-Moore theorem, implies Theorem~\ref{posisot}

\begin{thm}\label{main}
Let $M$ be a smooth, orientable, closed $n$-manifold such that
$\pi_1(M)$ is a free group on $k$ generators and $\pi_i(M)=0$ for
$2 \le \ i \ \le \frac {n}{2}$. If $n\neq 4$ or $k=1$, then $M$ is
homeomorphic to the connected sum of $k$ copies of $S^{n-1}\times
S^1$.
\end{thm}

Henceforth let $M$ be a smooth, orientable, closed $n$-manifold such
that $\pi_1(M)$ is a free group on $k$ generators and $\pi_i(M)=0$ for
$2 \le \ i \ \le \frac {n}{2}$. We assume throughout that all
manifolds we consider are orientable.

Let $\M$ be the universal cover of $M$. Hence $\pi_1(\M)$ is trivial
and so is $\pi_i(\M)=\pi_i(M)$ for $2 \le \ i \ \le \frac {n}{2}$. We shall show that the homology of $\M$ is isomorphic as $\pi_1(M)$-modules to that of the connected sum of $k$ copies of $S^{n-1}\times S^1$. We then show that $M$ is homotopy equivalent to the connected sum of $k$ copies of $S^{n-1}\times S^1$ using Theorems of Whitehead. Finally, recent results of Kreck and L\"uck allow us to conclude the result.

\begin{acknowledgements}
We thank the referees for helpful comments and for suggesting a simplification of our proof.
\end{acknowledgements} 

\section{The homology of $\M$}

Let $X$ denote the wedge $\vee_{j=1}^k S^1$ of $k$ circles and let $x$
denote the common point on the circles. Choose and fix an isomorphism
$\varphi$ from $\pi_1(M,p)$ to $\pi_1(X,x)$ for some basepoint $p\in
M$. We shall use this identification throughout. Denote $\pi_1(M,p)=\pi_1(X,x)$ by $\pi$.

As $X$ is an Eilenberg-Maclane space, there is a map $f:(M,p)\to
(X,z)$ inducing $\varphi$ on fundamental groups and a map $s:(X,z)\to
(M,p)$ so that $f\circ s:X\to X$ is homotopic to the
identity. 

We deduce the homology of $\M$ using the Hurewicz Theorem and Poincar\'e duality. 

\begin{lemma}\label{lowdim}
For $1\leq i\leq n/2$, $H_i(\M,\Z)=0$
\end{lemma}
\begin{proof}
As $\M$ is simply-connected and $\pi_i(\M)=\pi_i(M)=0$ for $1<i\leq n/2$ (by hypothesis), by Hurewicz theorem $H_i(\M,\Z)=0$ for $1\leq i\leq n/2$.
\end{proof}

We deduce the homology in dimensions above $n/2$ using Poincar\'e duality for $M$ with coefficients in the module $\Z[\pi]$, namely
$$H_{n-i}(M,\Z[\pi])=H^i(M,\Z[\pi])$$

Recall that $H_k(M,\Z[\pi])=H_k(\M,\Z)$ and the group $H^i(M,\Z[\pi])$ is the cohomology with compact support $H^i_c(\M,\Z)$. Hence Poincar\'e duality with coefficients in $\Z[\pi]$ is the same as Poincar\'e duality for a non-compact manifold relating homology to cohomology with compact support.

To apply Poincar\'e duality, we need the following lemma.

\begin{lemma}\label{isom}
For $1\leq i\leq n/2$, the map $s:(X,z)\to (M,p)$ induces isomorphisms of modules with $s_*:H^i(M;\Z[\pi])\to H^i(X;\Z[\pi])$.
\end{lemma}
\begin{proof}
As the map $s$ induces an isomorphism on homotopy groups in dimensions at most $n/2$, it induces isomorphisms on the cohomology groups with twisted coefficients. Specifically, we can add cells of dimensions $k\geq n/2+2$ to $M$ to obtain an Eilenberg-MacLane space $\bar{M}$ for the group $\pi$, which is thus homotopy equivalent to $X$. For $i\leq n/2$ and any $\Z[\pi]$-module $A$, it follows that 
$$H_i(M,A)=H_i(\bar{M},A)=H_i(X,A)$$
where the first equality follows as the cells added to $M$ to obtain $\bar{M}$ are of dimension at least $n/2+2$ and the second as the spaces are homotopy equivalent. 
\end{proof}
 
By applying Poincar\'e duality, we obtain the following result.

\begin{lemma}\label{homol}
Let $M$ be a smooth, orientable, closed $n$-manifold such that
$\pi_1(M)$ is a free group on $k$ generators and $\pi_i(M)=0$ for
$2 \le \ i \ \le \frac {n}{2}$. Then, for the universal cover $\M$ of $M$,
\begin{enumerate}
\item $H_i(\M,\Z)=0$ for $1\leq i<n-1$ 
\item We have an isomorphism $H_{n-1}(\M,\Z)=H^1_c(\widetilde{X},\Z)$, where $\widetilde{X}$ is the universal cover of $X$, determined by the isomorphisms $s_*:\pi_1(X,z)\to\pi_1(M,p)$ on fundamental groups. 
\end{enumerate}
\end{lemma}
\begin{proof}
The statements follow from Lemmas~\ref{lowdim} and~\ref{isom} by using $H_*(\M,\Z)=H_*(M,\Z[\pi])$.
\end{proof}

\section{Homotopy type}

We now show that $M$ is homotopy equivalent to the connected sum $Y$
of $k$ copies of $S^{n-1}\times S^1$. Our first step is to construct a
map $g:Y\to M$. We shall then show that it is a homotopy equivalence.

Note that $Y$ has the structure of a CW-complex obtained as
follows. The $1$-skeleton of $Y$ is the wedge $X$ of $k$ circles. Let
$\alpha_i$ denote the $i$th circle with a fixed orientation.

We attach $k$ $(n-1)$-cells $D_j$, with the $j$th attaching map mapping
$\del D^{n-1}$ to the midpoint $x_j$ of the $j$th circle. Finally, we
attach a single $n$-cell $\Delta$.

We associate to $D_j$ an element $A_j\in\pi_{n-1}(Y,x)$. Namely, as
the attaching map is constant, the $j$th $(n-1)$-cell gives an element
$B_j\in\pi_{n-1}(Y,x_j)$. We consider the subarc $\beta_j$ of $\alpha_j$
joining $z_j$ to $x$ in the negative direction and let $A_j$ be
obtained from $B_j$ by the change of basepoint isomorphism using
$\beta_j$.

Note that if we instead chose the arc joining $z_j$ to $x$ in the
positive direction, then the resulting element is $-\alpha_j\cdot
A_j$. By the construction of $Y$, it follows that the attaching map
of the $(n-1)$-cell represents the element
$$\del\Delta=\Sigma_j (A_j-\alpha_j\cdot A_j)$$
in $\pi_{n-1}(Y)$ regarded as a module over $\pi_1(Y)$. This can be
seen for instance by using Poincar\'e duality.

We now construct the map $g:Y\to M$. Recall that we have a map
$s:(X,z)\to (M,p)$ inducing the isomorphism $\varphi^{-1}$ on
fundamental groups. We define $g$ on the $1$-skeleton $X$ of $Y$ by
$g\vert_X=s$. We henceforth identify the fundamental groups of $Y$ and
$M$ using the isomorphism $\varphi$, i.e., $\pi_1(Y,z)$ is identified with $\pi$.

We next extend $g$ to the $n$-cell of $Y$ as follows. By Hurewicz theorem and Lemma~\ref{homol}, we have isomorphisms
of $\pi$-modules $\pi_{n-1}(M,p)=H_{n-1}(\M,\Z)$ and $\pi_{n-1}(Y,z)=H_{n-1}(\M,\Z)$. By Lemma~\ref{homol}, each of these modules is isomorphic to $H^1_c(\widetilde{X},\Z)$ with the isomorphisms determined by the identifications of the fundamental groups.

Under the above isomorphisms the elements $A_j$ correspond to
elements $A_j'$ in $\pi_{n-1}(M,p)$.  Consider the element $B_j'$ of
$\pi_{n-1}(M,g(z_j))$ obtained from $A_j'$ by the basechange map using
the arc $f(\beta_j)$. We define the map $g$ on $D_j$ extending the
constant map on its boundary to be a representative of $B_j$.

As the $\pi$-modules $\pi_{n-1}(M,p)$ and $\pi_{n-1}(Y,z)$ are
isomorphic, the image $g$ of $\del\Delta$ is homotopically
trivial. Hence we can extend the map $g$ across the cell $\Delta$.

\begin{lemma}
The map $g:Y\to M$ is a homotopy equivalence.
\end{lemma}
\begin{proof}
Let $G:\widetilde{Y}\to \M$ be the induced map on the universal covers. By Lemma~\ref{homol} applied to $M$ and $Y$, we see that
$H_p(\widetilde{Y})=H_p(\M)=0$ for $0<p\neq n-1$ and  $G$ induces an isomorphism on
$H_{n-1}$. Thus the map $G$ is a homology equivalence. By a theorem of
Whitehead~\cite{Wh}, a homology equivalence between simply-connected
CW-complexes is a homotopy equivalence.

It follows that $G$ induces isomorphisms $G_*:\pi_k(\widetilde{Y})\to
\pi_k(\M)$ for $k>1$. As covering maps induce isomorphims on higher
homotopy groups, and $g$ induces an isomorphism on $\pi_1$, it follows
that $g$ is a weak homotopy equivalence, hence a homotopy
equivalence(see~\cite{Ha}).
\end{proof}

\section{Proof of Theorem~\ref{main}}

The rest of the proof of Theorem~\ref{main} is based on results of
Kreck-L\"uck~\cite{KL}. In~\cite{KL}, the authors define a
manifold $N$ to be a \emph{Borel manifold} if any manifold
homotopy equivalent to $N$ is homeomorphic to $N$. We have shown
that a manifold $M$ satisfying the hypothesis of
Theorem~\ref{main} is homotopy equivalent to the connected sum
$Y$ of $k$ copies of $S^{n-1}\times S^1$. Hence it suffices to
observe that $Y$ is Borel.

By Theorem~0.13(b) of~\cite{KL}, the manifold $S^{n-1}\times S^1$
is Borel for $n \ge 4$. This completes the proof in the case when
$\pi_1(M)=\Z$. Further, if $n\geq 5$, then Theorem~0.9
of~\cite{KL} says that the connected sum of Borel manifolds is
Borel, hence $Y$ is Borel. This concludes the proof for
$\pi_1(M)$ a free group and $n  \geq 5$.\qed

Finally, in the case when $n=3$ by the Knesser conjecture (proved by Stallings) the manifold $M$ is a connected sum of manifolds whose fundamental group is $\Z$. As $M$ is orientable, it follows that if $M$ is expressed as a connected sum of prime manifolds (such a decomposition exists and is unique by the Knesser-Milnor theorem), then each prime component is either $S^2\times S^1$ or a homotopy sphere. By the Poincar\'e conjecture (Perelman's theorem), every homotopy $3$-sphere is homeomorphic to a sphere. It follows that $M$ is the connected sum of $k$ copies of $S^2\times S^1$.\qed

\end{document}